\documentclass[12pt,a4paper]{article}
\usepackage{latexsym,amssymb,amsmath,amsthm}
\usepackage{epsfig,graphicx,bm}
\usepackage[latin2]{inputenc}
\usepackage{color}
\usepackage{t1enc}
\usepackage{amssymb}
\usepackage{amsmath}
\newtheorem{thm}{Theorem}

\newtheorem{claim}[thm]{Claim}

\newtheorem{conjecture}[thm]{Conjecture}
\newtheorem{proposition}[thm]{Proposition}

\newtheorem{corollary}[thm]{Corollary}

\DeclareMathOperator{\cl}{cl}

\newcommand{\cP}{\mathcal P}

\title{Line Percolation in Finite Projective Planes}

\begin{document}

\author{
D\'aniel Gerbner\thanks{Research supported by the J\'anos Bolyai Research Fellowship of the Hungarian Academy of Sciences.} \thanks{Research supported by the National Research, Development and Innovation
Office -- NKFIH under the grant PD 109537.}
\and
Bal\'azs Keszegh\footnotemark[1]  \thanks{Research supported by the National Research, Development and Innovation
Office -- NKFIH under the grant PD 108406.} \footnotemark[4]
\and
G\'abor M\'esz\'aros \thanks{Research supported by the National Research, Development and Innovation
Office -- NKFIH under the grant K 116769.}
\and
Bal\'azs Patk\'os\footnotemark[1]  \thanks{Research supported by the National Research, Development and Innovation
Office -- NKFIH under the grant SNN 116095.}
\and 
M\'at\'e Vizer\footnotemark[5]}

\date{ 
\normalsize MTA R\'enyi Institute, Hungary H-1053, Budapest, Re\'altanoda utca 13-15.\\
\smallskip
\small \texttt{gerbner,keszegh,meszagab,patkos@renyi.hu, vizermate@gmail.com}\\
\smallskip 
\normalsize \today}

\maketitle

\begin{abstract}
We study combinatorial parameters of a recently introduced bootstrap percolation problem in finite projective planes. We present sharp results on the size of the minimum percolating sets and the maximal non-percolating sets. Additional results on the minimal and maximal percolation time as well as on the critical probability in the projective plane are also presented.
\end{abstract}

\section{Introduction}

Bootstrap percolation models have been frequently studied in the past decades. They offer widely utilizable models in different fields such as crack formation, clustering phenomena, the dynamics of glasses and sandpiles, or neural nets and economics \cite{Lee,Peres,Amini}. A new geometric bootstrap percolation model has been recently introduced and studied by Balister, Bollob\'as, Lee, and Narayanan \cite{gridpercolate}. The {\it  $r$-neighbor line percolation model} simulates the spread of an infection on the $d$-dimensional lattice $[n]^d$. The infection is carried by the axis parallel lines that pass through the lattice points of $[n]^d$; more precisely, if an axis parallel line $l$ contains $r$ infected points, then the entire point set of $l$ becomes infected. A subset $A$ of $[n]^d$ is called a {\it percolating set} if, by initially infecting the points of $A$ the infection spreads and infects every lattice point. In \cite{gridpercolate} Balister et al. proved that in the $r$-neighbor line percolation model a percolating set contains at least $r^d$ points. They also determined the order of magnitude of the critical probability (see the precise definition in section 7) of percolation when points of the initially infected sets are selected (independently) randomly with the same probability from $[n]^d$. 

The $r$-neighbor {\it bootstrap percolation models} on graphs has an ample literature. In particular, $r$-neighbor bootstrap percolation on $d$-dimensional complete grids had been studied by Balogh, Bollob\'as, Duminil-Copin, and Morris \cite{alldim} and on  Hamming tori by Gravner, Hoffman, Pfeiffer, and Sivakof \cite{earlier}. Note that the percolation in the above examples is modeled on some underlying graph. The new model of Balister et al. can be described by means of a more general incidence structure (i.e. hypergraph), namely, the incidences of the points and the axis parallel lines of the grid $[n]^d$. 
We generalize the model in \cite{gridpercolate} as follows: let $\mathcal{S}=(\mathcal{P},\mathcal{L}, \mathcal{R})$ be an incidence structure with point set $\mathcal{P}$, line set $\mathcal{L}$, and incidence relation $\mathcal{R}$ . The generalized line percolation model can be described as follows: let $A\subset\mathcal{P}$ be an initially infected set of points. The infection spreads in $\mathcal{S}$ along a line $l\in\mathcal{L}$ if it has at least $r$ infected points; in this case every point on that line becomes infected (we call such a line an "infected line").
That is, we define the sequence of subsets $\{A^{k}\}$  recursively with $A^0=A$ and for $s \ge 1$ let
\[A^{s}= A^{s-1}\cup \{P\in \mathcal{P}: \exists l\in\mathcal{L} ~\text{such that}\ P\in l, \ |l\cap A^{s-1}|\ge r\}.\]
We call the spread of infection from $A^{s-1}$ to $A^s$ {\it round $s$}. We call a set $A^k$ the {\it closure} of $A$ if $A^k=A^{k+1}$; we denote the closure of set $A$ by $\cl(A)$.
We say that a subset $A\subset \mathcal{P}$ percolates if $\cl(A)=\mathcal{P}$. We define the {\it time of the percolation} at the initial percolating set $A$ as the smallest $k\in\mathbb{N}$ for which $A^k=\mathcal{P}$. (For results on the time of percolation in some bootstrap percolation models, see e.g. \cite{maxtime2}, \cite{maxtime1},\cite{time}.) We call a percolating set $A$ minimal if none of its proper subsets percolate.

We are interested in the following parameters: the minimum size of a percolating set, the maximum size of a non-percolating set, the minimum time of percolation for a minimal percolating set, the maximum time of percolation, as well as the critical probability of the  percolation with infection parameter $r$. We denote the parameters by $m_r$, $M_r$, $t_r$, $T_r$, and $p_r$, respectively.

Sometimes it will be useful to consider the following equivalent definition of line percolation: a set $A\subseteq \cal{P}$ percolates if there exists a \textit{percolating sequence $l_1,l_2,\dots,l_{|\mathcal{L}|}$} of lines of $\mathcal{L}$ such that $|(A \cup \bigcup_{j<i}l_j)\cap l_i|\ge r$ for all $i=1,2,\dots, |\mathcal{L}|$. We will refer to this definition as the \textit{one-by-one model}, $l_i$ is the line infected in \textit{step $i$}, and $A_i=(A\cap l_i)\setminus \cup_{j<i}l_j$ is the \textit{part of $A$ needed in step $i$}.

In this work we study the above parameters in finite projective planes of order $q$. We always assume $r\geq 3$ as the cases $r=1,2$ lead to rather straightforward problems for all parameters. It turns out that the parameters show different behavior when the infection parameter is small compared to the order of the plane and when these two parameters are comparable. The exact definitions of "small" and "comparable" are presented at the discussion of the parameters. In general, our observation is that $r\approx \sqrt{2q}$ and $r\approx\frac{q}{2}$ are important milestones where most of the above parameters change behavior.

For arbitrary projective planes of order $q$ we establish the results collected in the table below. Furthermore we obtain stronger bounds for the standard projective (Galois) plane of order $q$ denoted by $PG(2,q)$ and coordinatized by the finite field $\mathbb{F}_q$.

\vspace{1cm}

\[
\begin{array}{|c|l|}
\hline
  \multicolumn{2}{|c|}{  } \\
  \multicolumn{2}{|c|}{ Results \ for \ \Pi_q } \\
  \multicolumn{2}{|c|}{  } \\
\hline \hline
& \\
& \bullet \ \binom{r+1}{2} \le m_r(\Pi_q)\leq (r-1)r+1 \ \ (\textrm{Proposition \ref{minpercupper}, Proposition \ref{maxPercConst}}) \\
m_r & \\
& \bullet \ \textrm{if} \ r<\sqrt{2q}, \ \textrm{then} \ m_r(\Pi_q)=\binom{r+1}{2} \ \ (\textrm{Proposition \ref{minPercConst}}) \\
& \\
& \bullet \ m_r(\Pi_q)= (1-o(1))q^2 \ \textrm{if} \ r=(1-o(1))q \ \textrm{and} \ q\rightarrow\infty \ \ (\textrm{Theorem \ref{minpercoq}})  \\
& \\
\hline
& \\
& \bullet \ q(r-1)+1\leq  M_r(\Pi_q)\leq (q+1)(r-1) \ \ (\textrm{Proposition \ref{gap}})   \\
M_r& \\
& \bullet \ \textrm{if} \ r< \frac{q}{2}+2, \ \textrm{then} \ M_r(\Pi_q)=q(r-1)+1 \ \ (\textrm{Proposition \ref{Mexact}})\\
& \\
\hline
& \\
t_r  & \bullet \ \textrm{for each }\ 4\le r \le \frac{q+7}{3}, \ \textrm{there is a minimal percolating set } A_r \subset \Pi_q \\
&  \hspace{2mm} \textrm{ with } t_r(A_r)=3 \ \ (\textrm{Proposition \ref{existsminperc}})\\
&  \\
\hline
& \\
T_r  & \bullet \ \textrm{if} \ r \ge 5 \textrm{ and } \binom{r}{2}\leq q, \ \textrm{then} \ T_r(\Pi_q) = r+1 \ \ (\textrm{Proposition \ref{maxperctime}, Proposition \ref{maxpercequal}})  \\
& \\
\hline

\end{array}
\]

\vspace{1cm}

In the probabilistic setting we consider the random subset $\Pi_q(p)$ that contains every point $P$ of $\Pi_q$ with probability $p$ independently of all other points. We determine the critical probability of $\Pi_q(p)$ to percolate. We also
consider a random process: if we pick points $P_1,P_2,\dots, P_{q^2+q+1}$ of $\Pi_q$ one-by-one such that $P_i$ is selected uniformly at random from $\mathcal{P}\setminus \{P_1,P_2,\dots, P_{i-1}\}$, then clearly the set $A_i=\{P_1,P_2,\dots, P_i\}$ does not percolate as long as $|A_i\cap l|<r$ for all $l\in\mathcal{L}$. We prove a bottleneck phenomenon: for the smallest index $i$ for which there exists a line $l \in \mathcal{L}$ with $|l\cap A_i|=r$, the set $A_i$ percolates with probability tending to 1 as $q$ tends to infinity.

\section{Folklore results and general observations}

We start with the recollection of some well known properties of finite projective planes in general as well as some peculiar characteristics of the standard projective planes. For a more detailed survey of the field we refer the reader to \cite{fingeosurvey}.
We denote an arbitrary finite projective plane of order $q$ ($q\geq 2$) by $\Pi_q$. Recall that $\Pi_q$ contains $q^2+q+1$ points and $q^2+q+1$ lines, that is, $|\mathcal{P}|=|\mathcal{L}|=q^2+q+1$. Every line $l$ contains $q+1$ points and every point is contained in $q+1$ lines; every pair of lines intersects in a unique point and every pair of points is contained in a unique line.

The dual plane $\overline{\Pi_q}=(\overline{\mathcal{P}},\overline{\mathcal{L}}, \overline{\mathcal{R}})$ of the finite plane $\Pi_q=(\mathcal{P},\mathcal{L}, \mathcal{R})$ is defined as follows: $\overline{\mathcal{P}}=\mathcal{L}$, $\overline{\mathcal{L}}=\mathcal{P}$, and $\overline{\mathcal{R}}=\mathcal{R}$ in the sense that a point $P$ and a line $l$ are incident in one plane if and only if their line-point pair is incident in the other. We mention that some projective planes - such as the standard plane - are self-dual, meaning that $\Pi_q$ and $\overline{\Pi_q}$ are isomorphic.

A set of $k$ points forms a {\it $k$-arc} if no three of them are collinear (i.e. contained in a line). An arc of a finite projective plane of order $q$ contains at most $q+2$ points if $q$ is even and at most $q+1$ points if $q$ is odd. Arcs containing $q+1$ and $q+2$ points are called {\it ovals} and ${\it hyperovals}$, respectively. Ovals and hyperovals can be easily found in standard projective planes. On the other hand, the existence of ovals and hyperovals in general planes is a longstanding open question. For general planes the best known bound is the following: 

\begin{proposition}\label{minarc}
If $S$ is an arc in $\Pi_q$ containing $k$ points such that $S$ is not contained by any other arc of the plane, then $q < \binom{k}{2}$.
\end{proposition}

A set of $k$ lines is called {\it lines in general position} if no three of them intersect at the same point. Obviously, every set of $k$ lines in general position in $\Pi_q$ corresponds to a $k$-arc in $\overline{\Pi_q}$, thus the above results concerning the possible size of arcs can be naturally translated to the size of sets of lines in general position.

We close up this section with some general observations that will be frequently used in our later proofs.

\begin{proposition}\label{notpercolate}
If $k\leq r-1$ and $A\subset\cup_{i=1}^k l_i$ for $l_i\in \mathcal{L}$, then $A$ does not percolate.
\end{proposition}
\begin{proof}
It is sufficient to show that the set $A'= \cup_{i=1}^k l_i$ does not percolate; we in fact show that $A'^1=A$. Assume by contradiction that a line $l$ gets infected in the first round. It implies that $|l\cap \cup_{i=1}^k l_i|\geq r$. Note on the other hand that $l\neq l_i$ thus $|l\cap l_i| = 1$ and so $|l\cap \cup_{i=1}^k l_i|\leq k\leq r-1$, contradicting our assumption.
\end{proof}
\begin{proposition}\label{r-broom}
If a set of points $A$ contains the point set of $r$ lines intersecting at a single point $P$ (an "$r$-broom"), then $A$ percolates.
\end{proposition}
\begin{proof}
Every line $l$ not containing $P$ intersects the $r$-broom at $r$ different points and thus gets infected in $A^1$.  
\end{proof}
\begin{proposition}\label{r-lines}
If $\binom{r}{2}\leq q$ and $A^i$ contains $r$ infected lines for some $i\in \mathbb{N}$, then $A^{i+1}=\mathcal{P}$.
\end{proposition}
\begin{proof}
Let $l_1,\dots,l_r$ be the $r$ infected lines and let $P\in \mathcal{P}$ not infected at the end of round $i$. Then, using $\binom{r}{2}\leq q$, $P$ is on a line $l$ not containing any intersection $l_i\cap l_j \ (1 \le i < j \le r)$, thus $|l\cap\bigcup_{i=1}^r l_i|= r$, hence $l$ gets infected in round $i+1$ and so does $P$.
\end{proof}
\section{Minimal percolating sets}
Proposition \ref{r-broom} immediately implies the following rather obvious upper bound: 
\begin{proposition}\label{minpercupper}
	$m_r(\Pi_q)\leq (r-1)r+1.$
\end{proposition} 
We believe that the presented bound is not sharp for $r\geq 3$. We first establish a general lower bound on $m_r(\Pi_q)$.

\begin{proposition}\label{maxPercConst}
	$m_r(\Pi_q)\geq \binom{r+1}{2}$. Moreover, if $m_r(\Pi_q) = \binom{r+1}{2}$, then the points of $A$ are contained in the union of $r$ lines in general position. 
\end{proposition}
\begin{proof}
We consider the one-by-one model and observe that if $A$ percolates, then for any percolating sequence of lines for $i=1,2\dots,r$ the part $A_i$ of $A$ needed at step $i$ has size at least $r-i+1$, hence $|A|\geq\binom{r+1}{2}$.
	
It immediately follows that if $m_r(\Pi_q) = \binom{r+1}{2}$ and $A$ is a smallest percolating set with percolating sequence $l_1,l_2,\dots$, then $|A_i|=r-i+1$ and all intersection points $l_i\cap l_j$ ($1\le i<j\le r$) are distinct. Therefore the lines $l_1,l_2,\dots, l_{r}$ form a set of $r$ lines in general position.
\end{proof}
Proposition \ref{maxPercConst} shows that if $m_r(\Pi_q) = \binom{r+1}{2}$ holds, the plane $\Pi_q$ must contain at least $r$ lines in general position. It is only known that such a set of lines exists in every projective plane for $r\leq \sqrt{2q}$ (Proposition \ref{minarc}). We show, on the other hand, that a sufficiently large set of lines in general position can provide a percolating set of size $\binom{r+1}{2}$, thus the bound of Proposition \ref{maxPercConst} is sharp in this case:

\begin{proposition}\label{dual}
Suppose that there exists a set $L$ of $k\geq 2r$ lines in general position in $\Pi_q$. In this case we have $m_r(\Pi_q)=\binom{r+1}{2}$.
\end{proposition}
\begin{proof}
Let $L=\{l_1,\dots ,l_{k}\}$ and consider any sequence of all lines of $\cal{L}$ that starts with $l_1,l_2,\dots,l_k$. We construct a percolating set $A$ with the above percolating sequence by defining the parts $A_i$ needed at step $i$. As $L$ is a set of lines in general position, we have $|l_i\cap (\cup_{j<i}l_j)|=i-1$, therefore for $i\le r$ let $A_i$ be any $(r-i+1)$-subset of $l_i\setminus (\cup_{j<i}l_j)$.
We claim that for $i> r$ we can let $A_i$ be empty. Indeed, if $r< i\le k$, then as $L$ is a set of lines in general position, we have that $l_i\cap (\cup_{j<i}l_j)\supseteq l_i\cap (\cup_{j\le r}l_j)$ has size at least $r$. While if $i> k$, then, again using that $L$ is a set of lines in general position, we have that $|l_i\cap (\cup_{j<i}l_j)|\ge |l_i\cap (\cup_{j\le k}l_j)|\ge k/2\ge r$.
\end{proof}

\begin{corollary}\label{m_r} We have the following:

    \begin{itemize}
		\item[i)] $m_r(\Pi_q)=\binom{r+1}{2}$ for $r< \sqrt{\frac{q}{2}},$
		
		\item[ii)] $m_r(PG(2,q))=\binom{r+1}{2}$ for $r\leq\lfloor\frac{q+1}{2}\rfloor$.
	\end{itemize}
\end{corollary}
\begin{proof}
	In both cases we want to apply Proposition \ref{dual}.
	
	Proposition \ref{minarc} implies that for a maximal $k$-arc in $\overline{\Pi_q}$ we have $q < \binom{k}{2}$. This implies that $\overline{\Pi_q}$ contains a $2r$-arc if $r< \sqrt{\frac{q}{2}}$, hence $\Pi_q$ has $2r$ lines in general position. For the second statement recall that the standard projective plane is self-dual and contains ovals, thus it contains $q+1$ lines in general position.
\end{proof}

We slightly improve Corollary \ref{m_r} i) in the next proposition:

\begin{proposition}\label{minPercConst}
	If $r< \sqrt{2q}$, then $m_r(\Pi_q) = \binom{r+1}{2}$.
\end{proposition}
\begin{proof}
We use the one-by-one model and construct simultaneously a percolating sequence of lines and a percolating set $A$ by defining its needed parts $A_i$. Proposition \ref{minarc} and the assumption $r< \sqrt{2q}$ imply that there exists a set $L=\{l_1,l_2,\dots, l_{r}\}$ of $r$ lines in general position. As in the proof of Proposition \ref{dual}, we start our percolating sequence with $l_1,l_2,\dots,l_{r}$ and let $A_i$ be an $(r-i+1)$-subset of $l_i\setminus \cup_{j<i}l_j$ and we let $A_i$ be empty for $i> r$. We have to finish our percolating sequence.

Let us choose $P\in l_1\setminus \cup_{j=2}^{r}l_j$ and let $l_{r+1},l_{r+2},\dots, l_{2r-1}$ be lines containing $P$
but not containing any intersection $l_i\cap l_j$ ($2\le i<j\le r$). Observe that $\binom{r-1}{2}+1+(r-2) < q+1$ enables us to choose such lines. As for every $i=r+1,r+2,\dots,2r-1$ the line $l_i$ intersects $\cup_{j=0}^{r-1}l_j$ in $r$ different points, we can continue our percolating sequence by $l_{r+1},l_{r+2},\dots,l_{2r-1}$. As $l_1,l_{r+1},l_{r+2},\dots,l_{2r-1}$ form an $r$-broom, we can first add all lines not containing $P$ to our percolating sequence and then all remaining lines.
\end{proof}

By a more elaborate use of the one-by-one model we strengthen further the lower bound on $m_r(\Pi_q)$. Let $N$ and $j$ ($j\leq q$) be positive integers to be set later.
Let $A$ be a smallest percolating subset of $\Pi_q$ and let $l_1,\dots, l_{q^2+q+1}$ be a percolating sequence of lines. We will consider the initial segment of size $N$ of the percolating sequence and study the intersection of the $i$\textsuperscript{th} line with the union of the previous lines as follows: let $f^k(i)$ denote the number of points of $l_{i+1}$ that are adjacent to exactly $k$ lines out of $\{l_1,\dots, l_i\}$ for $k=1,2,\dots,j$   and let $g(i)$ denote the number of points of $l_{i+1}$ with more than $j$ adjacencies. Obviously, $f^1(1)=\dots=f^j(1)=g(1)=0$.  In addition, let $s^k(i)$ denote the number of all points adjacent to exactly $k$ lines  and let $\hat{s}(i)$ denote the number of points with more than $j$ adjacent lines out of $\{l_1,\dots, l_i\}$ for $k=1,2,\dots,j$. We define $f^k(0)=0$, $s^k(0)=0$ for  $k=1,2,\dots,j$ and $g(0)=0$ and $\hat{s}(0)=0$. Observe that the following identities hold for $0\le i\le N-1$:

$$	s^1(i+1) - s^1(i)= q+1 - \sum\limits_{k=1}^{j}f^k(i) - g(i) - f^1(i), $$ 
$$
	s^k(i+1)-s^k(i)  = f^{k-1}(i) - f^k(i)\text{ } \ \ (2\leq k\leq j), 
$$

Introducing the notation $f_k=\sum\limits_{i=1}^{N-1} f^k(i)$ and $g=\sum\limits_{i=1}^{N-1}g(i)$ and summing up our above equations for $i=0,1,\dots,N-1$ we obtain:
$$
	s^1(N) - s^1(0)  = N(q+1) - \sum\limits_{k=1}^{j}f_k - f_1-g, $$ 
	$$	s^k(N) - s^k(0)  = f_{k-1} - f_k\text{ } \ \ (2\leq k\leq j), 
$$
Observe that the left hand sides of the above equations are non-negative, thus we have obtained
\begin{align}
	0&\leq N(q+1) - \sum\limits_{k=1}^{j}f_k - f_1-g, \label{eq1}\\
	0&\leq f_{k-1} - f_k\text{ } \ \ (2\leq k\leq j), \label{eq2}
\end{align}

By definition $f_j$ and $g$ are the sums of non-negative numbers, thus:
\begin{align}
0&\leq f_{j}.\label{eq5}\\
0&\le g.\label{eq6}
\end{align}

Counting the adjacencies of $l_{i+1}$ with the previous $i$ lines, we obtain the following inequality:
$$
	\sum\limits_{k=1}^j k f^k(i)+(j+1) g(i)\leq i.
$$
By summing again for $i=0,1,\dots,N-1$ we get
\begin{equation}
	\sum\limits_{k=1}^j k f_k+(j+1) g\leq \binom{N}{2}.\label{eq4}
\end{equation}

Finally, observe that for $A_i$, the needed part of $A$ in step $i$, we have \[|A_i|\ge \max\{0, r-\sum\limits_{k=1}^j f^k(i)-g(i)\}\geq r-\sum\limits_{k=1}^j f^k(i)-g(i). \] This yields:
\begin{equation}\label{opt}
	m_r(\Pi_q)=|A|\ge\sum_{i=1}^N|A_i|\geq Nr-\sum\limits_{k=1}^j f_k-g.
\end{equation}
Therefore we want to minimize the $(j+1)$-variable function $h:=Nr-\sum_{k=1}^j f_k-g$ on the solutions $(f_1,f_2,\dots,f_j,g)\in     \mathbf{R}^{d+1}$ of the  inequality system $(\ref{eq1},\ref{eq2},\ref{eq5},\ref{eq6},\ref{eq4})$ of $j+3$ inequalities, which is a convex polytope $P$ in $\mathbf{R}^{d+1}$. 

The reason is that we know by (\ref{opt}) that $m_r(\Pi_q)\ge h$ for some values of the variables in $P$, and so a global minimum of $h$ on $P$ is a lower bound on $m_r(\Pi_q)$.

Now any linear function is minimized on some vertex of the polytope. All vertices are such that $j+1$ of the inequalities hold with  equality (and the additional $2$ of them hold with equality or inequality).

We claim that the optimal vertex is the point $v$ for which the  $j+1$ inequalities (\ref{eq1}),(\ref{eq2}),(\ref{eq4}) hold with equality, assuming that this is indeed a vertex of $P$.

Take the hyperplane $H$ going through $v$ and on which $h$ is constant. 
Observe that $\underline{0}\in P$ and $h(\underline{0})=Nr$ while for $v$ (\ref{eq1}) holds with equality and so $h(v)=Nr-N(q+1)+f_1<Nr$ (as $f_1$ is the sum of $N$ numbers, all at most $q+1$ and the first summand is $f^1(0)=0<q+1$). Thus, $v$ cannot maximize $h$ on $P$. So to show that $v$ minimizes $h$ on $P$, it is enough to prove that $v$ is the only point in $H\cap P$.

Take any vector parallel to $H$, that is $(d_1,d_2,\dots, d_j,d_{j+1})\ne \underline{0}$ such that $\sum d_i=0$. We claim that $v+d$ is outside $P$. Indeed, recall that for $v$ inequalities (\ref{eq1}),(\ref{eq2}),(\ref{eq4}) hold with equality, and $\sum d_i=0$. Thus, assuming $v+d$ is inside $P$, by (\ref{eq1}) for $v+d$ we must have $d_1\le 0$, while from inequality system (\ref{eq2}) we get that $$d_j\le d_{j-1}\le \dots \le d_1\le 0.$$
By (\ref{eq4}) we get $$d_1+2d_2+\dots jd_j+(j+1)d_{j+1}\le 0$$ and using that $d_{j+1}=-d_1-d_2-\dots -d_j$ this gives 
$$-jd_1-(j-1)d_2-\dots -d_j\le 0.$$ These imply that $d_i=0$ for all $i=1,2,\dots,j+1$, a contradiction. We can conclude that $v$ minimizes $h$ on $P$.

For the vertex $v$ we have that $f_1=f_2=f_j=f$ for some number $f$ as inequality system (\ref{eq2}) holds with equality and then as inequalities (\ref{eq1}),(\ref{eq4}) hold with equality too, also $$(j+1)f+g=N(q+1)$$  $$N(N-1)/2=(j+1)(jf/2+g). $$ Solving these two equalities we get that for $v$ we have 

$$	f=\frac{2}{j+2}N(q+1)-\frac{N(N-1)}{(j+2)(j+1)},$$
$$	g=\frac{N(N-1)}{j+2}-\frac{j}{j+2}N(q+1),$$
	\begin{align}
	h_{min}=Nr-jf-g=Nr-N(q+1)\frac{j}{j+2}-\frac{N(N-1)}{(j+1)(j+2)}.\label{hbound}
\end{align}

Earlier we assumed that $v$ is not outside $P$, that is, it is a vertex of $P$. For that we need that (\ref{eq5}) and (\ref{eq6}) hold, that is, $f\ge 0$ and $g\ge 0$. This gives that we need:
\begin{align}
	j(q+1) \le N-1\le 2(j+1)(q+1).\label{inside}
\end{align}

Now we can use (\ref{hbound}) to get explicit lower bounds on $m_r(\Pi_q)$ as by (\ref{opt}) we have $m_r(\Pi_q)\ge h\ge h_{min}$.

\smallskip

We first consider the case $r=(1-o(1))q$. For fixed $j$ let $N=(j+1)q$ which clearly satisfies (\ref{inside}). Then (\ref{hbound}) gives the following 
lower bound on $m_r(\Pi_q)$:
\[
m_r(\Pi_q)\ge q^2(1-o(1))(j+1)-q^2\frac{j(j+1)}{j+2}-q\frac{j(j+1)}{j+2}-q^2\frac{j+1}{j+2}+q\frac{1}{j+2}= \frac{j+1}{j+2}q^2-o(q^2).
\]

As $j$ can be chosen arbitrarily (only $j<q$ was assumed), we obtain the following result if $q\rightarrow \infty$ (as trivially $q^2+q+1$ is an upper bound). 
\begin{thm}\label{minpercoq}
	$m_{(1-o(1))q}(\Pi_q)= (1-o(1))q^2.$
	\end{thm}
\smallskip

In particular, $m_q(\Pi_q)= q^2+O(q)$.

Consider now the general case when $r=cq$ for some constant $0<c<1$.  For fixed $j$ let again $N=(j+1)q$ in order to satisfy (\ref{inside}). Then (\ref{hbound}) gives the following 
lower bound on $m_r(\Pi_q)$:

\[
m_r(\Pi_q)\ge \left(c(j+1)-\frac{(j+1)^2}{j+2}\right)q^2-O(q).
\]

For a given $c$, let $$t(c,j)=c(j+1)-\frac{(j+1)^2}{j+2}.$$  If $c\ge c_j=1-\frac{1}{(j+1)(j+2)}$ then as $t(c,j)$ is monotone in $c$, we get that $$t(c,j)\ge  t({c_j},j)= \frac{j}{j+2}.$$ 
Thus for general $c=1-\alpha$ this implies the following lower bound:

\begin{thm}\label{thm:cq}
	Given a constant $\alpha$ with $\frac{1}{ (j+2)(j+3)}< \alpha\le \frac{1}{(j+1)(j+2)}$ for some $j$ positive integer, we have 
	$$m_{(1-\alpha)q}(\Pi_q)\ge \left(\frac{j}{j+2}-o(1)\right)q^2\ge (1-2\sqrt{\alpha}-2\alpha-o(1))q^2.$$
\end{thm}

Apart from the $o(q^2)$ error term, Theorem \ref{thm:cq} gives the lower bound $(1-2\sqrt{\alpha}-2\alpha-o(1))q^2$. 
Recall that Proposition \ref{maxPercConst} gives the lower bound $(1/2-\alpha+\alpha^2/2)q^2$. The bound of Theorem \ref{thm:cq} is better than this when $\alpha<0.05$. More importantly, the bound of Theorem \ref{thm:cq} tends to $q^2$ as $\alpha\rightarrow 0$.
Finally, in this context Proposition \ref{minpercupper} gives the upper bound $(1-2\alpha+\alpha^2)q^2$.

\section{Maximal non-percolating sets}
Let $A\subset\mathcal{P}$ be a maximal non-percolating set; observe that $A^1=A$ necessarily holds, in other words, every non-infected line $l$ of the plane intersects $A$ in at most $r-1$ points. The focus of study in this section is the maximum number of points in $A$.

Note that similar structures of projective planes have been studied for a while; a set of point $A$ of $\Pi_q$ is called a $(k,n)$-\textit{arc} if $|A|=k$ and $A$ intersects every line of the plane in at most $n$ points (and it intersects at least one line in exactly $n$ points). It was proved by Barlotti \cite{Barlotti} that a $(k,n)$-arc of $\Pi_q$ contains at most $nq-q+n$ points. Obviously, a $(k,r-1)$-arc of the plane is a non-percolating set; applying substitution the above theorem yields that a non-percolating set obtained by a $(k,n)$-arc contains at most $(r-1)q-q+r-1$ points. On the other hand, non-percolating sets, unlike $(k,n)$-arcs, might contain fully infected lines as well, thus the above upper bounds on the $(k,n)$-arc are unlikely to provide sharp results for non-percolating sets; in fact we provide examples of non-percolating sets with $(r-1)q + 1 > (r-1)q-q+r-1$ points in the upcoming Proposition \ref{gap}. We also establish a general upper bound on the maximum size of a non-percolating set then investigate the sharpness of the bounds.
\begin{proposition}\label{gap}
$q(r-1)+1\leq  M_r(\Pi_q)\leq (q+1)(r-1)$.
\end{proposition}
\begin{proof}
For the first inequality, note that an $(r-1)$-broom does not percolate and has exactly $q(r-1)+1$ points. For the second inequality, take an uninfected point $P$ and observe that any line through $P$ contains at most $r-1$ infected points.
\end{proof}

\begin{proposition}\label{Mexact}
If $r< \frac{q}{2}+2$, then $M_r(\Pi_q) = q(r-1)+1$.
\end{proposition}
\begin{proof}
Let $A$ be a maximum size non-percolating set and let $P\in A$.
\begin{itemize}
\item[Case 1.] If there exists no infected line containing $P$, then every line through $P$ contains at most $r-2$ infected points, thus $|A|\leq (q+1)(r-2)+1\leq q(r-1)+1$.
\item[Case 2.] If there exists exactly one infected line containing $P$, then by similar counting we have
$|A|\leq q(r-2)+q+1 = q(r-1)+1$.
\item[Case 3.] In the remaining case every $P\in A$ is contained in at least 2 infected lines. Fix an infected line $l$ in $A$ and one additional infected line for every point in $l$. The union of the selected $q+2$ lines contains at least $(q+1)+q+(q-1)+\dots+1=\binom{q+2}{2}$ infected points which exceeds the upper bound of Proposition \ref{gap} if $r< \frac{q}{2}+2$, a contradiction.
\end{itemize}
\end{proof}

We have not been able to improve the upper bound of Proposition \ref{gap} for $r\geq \frac{q}{2}+2$. On the other hand, we show that the upper bound of $(q+1)(r-1)$ is sharp in certain cases. We believe that the results are particular cases of a more general pattern that holds for $r\geq \frac{q}{2}+2$ in general. We have been able to neither verify nor disprove the statement, thus we close up this section by stating it as a conjecture.

\begin{proposition}
Let $q$ be even and assume that the dual plane of $\Pi_q$ has a hyperoval. If $r=\frac{q}{2}+2$, then $M_{r}(\Pi_q) = (q+1)(r-1)$.
\end{proposition}
\begin{proof}
The hyperoval of the dual plane translates to $q+2$ lines in $\Pi_q$ such that no three intersect at the same point. Let $A$ be the union of those lines. Observe that every point is contained in exactly zero or two of those lines, thus any additional line intersects $A$ in exactly $\frac{q+2}{2}=\frac{q}{2}+1$ points, hence $A$ does not percolate.
\end{proof}
We mention that the above construction works for higher values of $r$ as well. Nevertheless, if $r> \frac{q}{2}+\frac{5}{2}$, the dual hyperoval  does not yield a bigger construction than the already discussed $(r-1)$-broom.
\begin{proposition}
Let $q$ be even and assume that $\Pi_q$ has a hyperoval.
 If $r=q$, then $M_r(\Pi_q) = (q+1)(r-1)$.
\end{proposition}
\begin{proof}
Let $A$ be a complement of the hyperoval in $\Pi_q$, it contains $(q+1)(q-1)$ points, as required. Further, every line on the plane contains exactly 2 non-infected points, thus $A$ does not percolate.
\end{proof}
\begin{conjecture}
If $r\geq \frac{q}{2}+2$, then $M_r(\Pi_q)=(q+1)(r-1)$.
\end{conjecture}

\section{Minimal percolation time}

Obviously, any percolating set $A$ can be extended to a percolating set $A'$ with percolation time $t_r(A')= 1$. Thus we are interested only in the minimal percolation time of inclusion minimal percolating sets. First we show that minimal percolating sets cannot fill up the entire plane within a single round:
\begin{proposition}
If $A$ percolates in one round, that is, $A^1=\mathcal{P}$, then $A$ cannot be minimal.
\end{proposition}
\begin{proof}
Assume on the contrary that $A\subset\mathcal{P}$ is minimal and percolates in one round. Let $P\in A$ be arbitrarily chosen and let $A'= A\setminus \{P\}$. Apparently, $P\not\in\cl(A')$, thus every line $l$ through $P$ contains at most $(r-1)$ additional points of $cl(A')$. 

Now choose an arbitrary line $l$ through $P$ and a $Q\in l$ with $Q\ne P$ and $Q\not\in\cl(A')$ (notice that this implies that $Q\notin A$). As $Q\in A^1$, we know that some line $l_1$ containing $Q$ got infected in the first round. Clearly $l_1=l$, as otherwise $l_1$ (and thus $Q$ too) would get infected also when starting with $A'$ instead of $A$. That implies that $l$ contains exactly $r$ points from $A$, including $P$.

We have concluded that every line through $P$ contains $r-1$ additional initially infected points. Observe that the set of points on any $r$ of these lines percolates, contradicting the minimality of $A$; this completes our proof.
\end{proof}

Observe that for $r=2$ all minimal percolating sets consist of 3 non-collinear points, thus $t_2(\Pi_q)=2$. Also, it is fairly easy to find a set of 6 points that percolates in two rounds for $r = 3$ (such a set must be minimal due to Proposition \ref{maxPercConst}), that means $t_3(\Pi_q)=2$ for all possible $q$. However we cannot prove that $t_r(\Pi_q)=2$ for higher values of $r$ thus state it as a conjecture at the end of this section. We show that a slightly weaker similar pattern holds for higher values of $r$. We note that if $r \le \sqrt{2q}$ then it is relatively easy to construct a minimal percolating set $A_r$ with $t_r(A_r)=3$, however we can prove the same when $r$ is linearly big in $q$:

\begin{figure}
\begin{center}
\includegraphics[width=8cm]{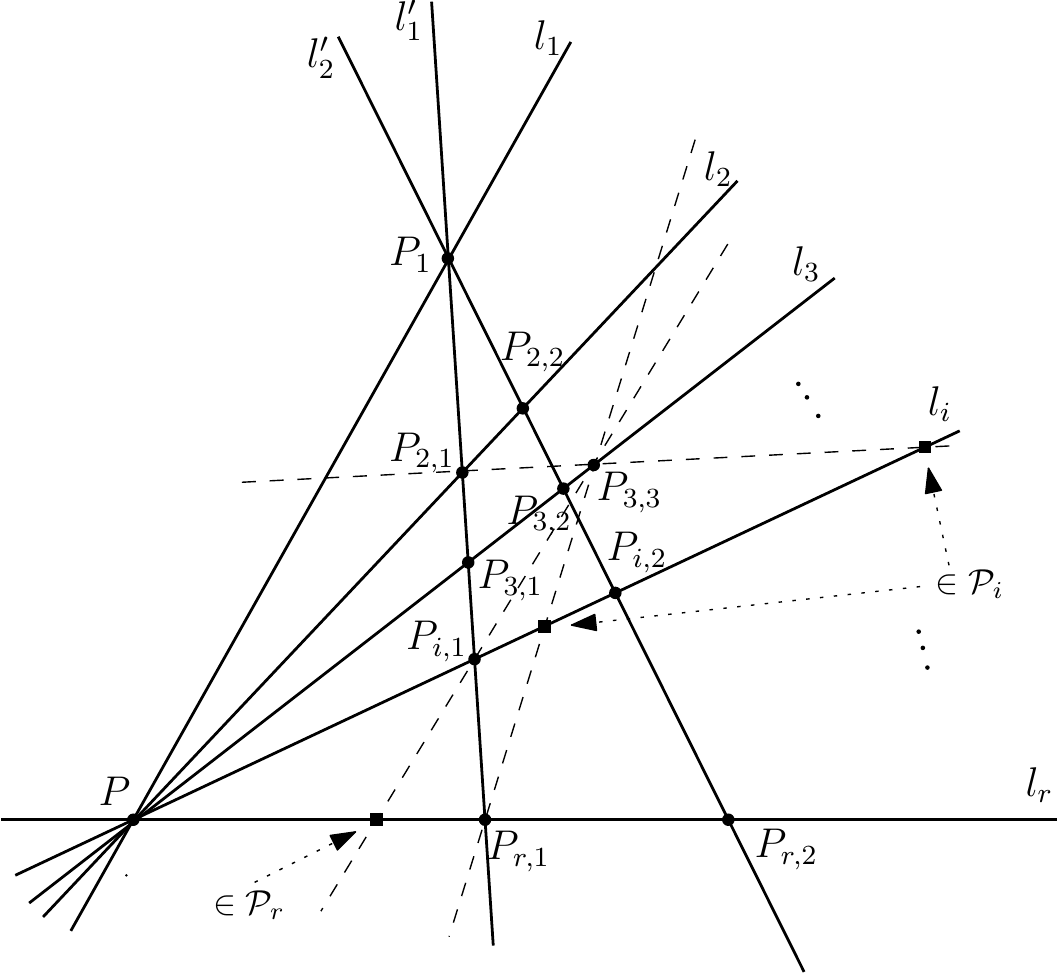}
\caption{A minimal set that percolates in $3$ rounds.}
\label{fig:constructionMinPerc}
\end{center}
\end{figure}

\begin{proposition}\label{existsminperc}
For every $4\leq r \leq \frac{q+7}{3}$ there exists a minimal percolating set $A_r$ such that $t_r(A_r)=3$.
\end{proposition}
\begin{proof}
Let $l_1,\dots,l_r$ be $r$ lines forming an $r$-broom with center $P$.
We construct the initially infected set $A_r$ as follows (see also Figure \ref{fig:constructionMinPerc}): let $P_1 \in l_1\setminus \{P\}$, and $P_{2,1},P_{2,2} \in l_2 \setminus \{P\}$ different. Let us denote $\overline{P_1P_{2,1}}$ by $l'_1$ and $\overline{P_1P_{2,2}}$ by $l'_2$. For any $3 \le i \le r$ let  $P_{i,1}=l_i\cap l'_1, P_{i,2}=l_i\cap l'_2$ and let $P_{3,3}\in l_3\setminus \{P,P_{3,1},P_{3,2}\}$ be arbitrary. For $4\le i\le r$ let 
$$\cP_i= \{l_i\cap \overline{P_{3,3}P_{j,1}}: 2\le j\le r, j\neq 3,j\neq i\}.$$
Note that $|\cP_i|=r-3$ for all $4\le i \le r$.

\vspace{2mm}

\noindent
Let $$A_r:=\{P,P_1,P_{3,3}\}\cup \{P_{i,1},P_{i,2}: 2\le i \le r\}\cup \bigcup_{i=4}^r\cP_i.$$

First we prove that $A_r$ percolates in three rounds. Observe that $$A_r^1 \supset \bigcup_{i=4}^{r} l_i\cup l'_1\cup l'_2 \cup \{P_{3,3}\}.$$
 We claim that $A_r^2$ contains an $r$-broom with center $P_{3,3}$. Any line $l$ with $P_{3,3}\in l$, $P,P_1 \notin l$ gets infected in round 2, if it intersects $l'_1$ and $l'_2$ outside $\bigcup_{i=4}^{r} l_i$. The number of such lines is at least $q+1-2-2(r-4)\ge r$ by the condition on $r$. Therefore $A_r^3=\cP$ by Proposition \ref{r-broom}.

\vspace{4mm}

The next claim shows that lines of special form cannot intersect certain subsets of $A_r$ in a big set. 

\begin{claim}\label{klaja}
For any line $l=\overline{P_{3,3},Q}$ with $Q\in (A_r\setminus \{P\}) \cap (\bigcup_{i=4}^rl_i)$, we have $$|l \cap (\bigcup_{i=3}^rl_i \cup l'_1 \cup l'_2)|\le r-1.$$
\end{claim}

\begin{proof}[Proof of Claim.]
As $\bigcup_{i=3}^rl_i \cup l'_1 \cup l'_2$ is a union of $r$ lines and $l$ is different from these lines, the intersection has size at most $r$. If the intersection was of size $r$, then $l$ would intersect all lines in $\{l_i: 3 \le i \le r \} \cup \{l'_1, l'_2\}$ in distinct points. But this is impossible as if for any $4 \le i \le r$ we have $Q=P_{i,1}$ (or $Q=P_{i,2}$), then $l\cap l'_1$ and $l\cap l_i$ ($l\cap l'_2$ and $l\cap l_i$, respectively) are the same. While if $Q\in \cP_i$, then $l\cap l'_1$ and $l\cap l_j$ are the same for some $4\le j\le r$ by the definition of $\cP_i$.
\end{proof}

Note that $A_r$ is contained in $r-1$ lines $l'_1,l'_2,l_4,l_5,\dots, l_r$ apart from $P_{3,3}$. Thus by the above claim in the first round only these $r-1$ lines can get infected, and in the second round only lines containing $P_{3,3}$ can get infected.
Now we want to prove that $A_r'=A_r\setminus \{R\}$ does not percolate for any $R\in A_r$. We proceed by a case analysis:

\vspace{2mm}

\textbf{Case 1:} $ \ R=P.$

\vspace{1mm}

In this case we have $A_r'^1=A_r'\cup l'_1 \cup l'_2$. This again shows that $A_r'^1$ is contained in $r-1$ lines $l'_1,l'_2,l_4,l_5,\dots, l_r$ apart from $P_{3,3}$. Thus any line that gets infected in the second round contains $P_{3,3}$ and intersects all these lines in infected points. In particular it intersects $l_4$ in an infected point $Q\in (A_r\setminus \{P\}) \cap (\bigcup_{i=4}^rl_i)$. Such lines contain at most $r-1$ infected points by Claim \ref{klaja}, which means that $A_r'$ does not percolate.

\vspace{2mm}

\textbf{Case 2:} $\ R\in \{P_1,P_{2,1},P_{2,2},P_{3,1},P_{3,2},P_{3,3}\}.$ 

\vspace{1mm}

We claim that $A_r \setminus \{R\}$ is contained in the union of $r-1$ lines thus by Proposition \ref{notpercolate} it cannot percolate. Indeed, $A_r\setminus \{P_1\}\subseteq \cup_{i=2}^rl_i$, $A_r\setminus \{P_{3,3}\}\subseteq l'_1\cup l'_2\cup \cup_{i=4}^rl_i$, $A_r\setminus \{P_{2,1}\} \subseteq l'_2\cup \cup_{i=3}^{r}l_i$, $A_r\setminus \{P_{2,2}\}\subseteq l'_1\cup \cup_{i=3}^{r}l_i$, $A_r\setminus \{P_{3,1}\}\subseteq l'_2\cup \overline{P_{3,3}P_{2,1}} \cup \cup_{i=4}^rl_i$,  and finally $A_r\setminus \{P_{3,2}\}\subseteq l'_1\cup \overline{P_{3,3}P_{2,2}} \cup \cup_{i=4}^rl_i$.

\vspace{2mm}

\textbf{Case 3:} $R\in (A\setminus \{P\}) \cap (\bigcup_{i=4}^r l_i$).

\vspace{1mm} In this case $R \in l_i \setminus \{P\}$ for some $4\le i \le r$, and $l_i$ does not get infected in the first round. $A_r'^1$ is contained in $r-1$ lines $l'_1,l'_2,l_4,l_5,\dots, l_r$ apart from $P_{3,3}$. Similarly to Case 1 any line that gets infected in the second round contains $P_{3,3}$ and intersects all these lines in infected points. In particular it intersects $l_i$ in an originally infected point $Q\in (A_r\setminus \{P\}) \cap (\bigcup_{i=4}^rl_i)$. Such lines contain at most $r-1$ infected points by Claim \ref{klaja}, which means that $A_r'$ does not percolate.

\end{proof}

\begin{conjecture}

For each $r$ with $4 \le r \le q$ there exists a minimal percolating set $A_r$ with $t_r(A_r) =	 2$.

\end{conjecture}

\section{Maximal percolation time}

Obviously, $T_1(\Pi_q)= 1$ and $T_2(\Pi_q)=2$. It is fairly easy to see that $T_3(\pi_q)=3$ as there should be $3$ not collinear points outside  any line $l$ that was infected in the first round since otherwise the set would not percolate by Proposition \ref{notpercolate}. Then in the second round all 3 lines through these points are infected and by Proposition \ref{r-lines} (using that $\binom{3}{2}=3$) we are done. If $r=4$, then similar argument shows that $T_4(\Pi_q)=4$ for $q \ge 6(=\binom{4}{2})$. 

However for $r\geq 5$ the situation is changing a bit. 

\begin{proposition}\label{maxperctime}
If $r \ge 5$ and $\binom{r}{2}\leq q$,  then we have $T_r(\Pi_q)\geq r+1$.
\end{proposition}
\begin{proof}
We construct a percolating set $A_r$ with percolation time at least $r+1$ as follows. Using that $r \ge 5$ we can choose $r$ lines, $l_1,l_2,...,l_r$ in general position, and point sets $A_{r,i} \subset l_i$ for $i=1,2,...,r$ with the following properties (see Figure \ref{fig:constructionAr}):

\vspace{2mm}

1. $|A_{r,i}|=r+1-i,$

\vspace{1mm}

2. for any $1\le i < j \le r$ we have $l_i \cap l_j \not \in \cup_{i=1}^r A_{r,i},$

\vspace{1mm}

3. (using the notation $A_{r,r}=\{P\}, \ A_{r,r-1}=\{Q_1,Q_2\}$ and $q_1=\overline{PQ_1}, q_2=\overline{PQ_2}$,) 

\hspace{4mm} $l_1 \cap l_2 \in q_1$, $l_1 \cap l_3 \in q_2$,

\vspace{1mm}

4. $A_{r,r-2} \cap (q_1 \cup q_2) = \emptyset$.

\vspace{2mm}

We can easily choose such sets using the conditions on $r$. Let $$A_r:=\bigcup_{i=1}^r A_{r,i}.$$ 

The next claim states that in the $j^{th}$ round ($j \le r-1$) only $l_i$ is the new infected line.

\begin{figure}
\begin{center}
\includegraphics[width=10cm]{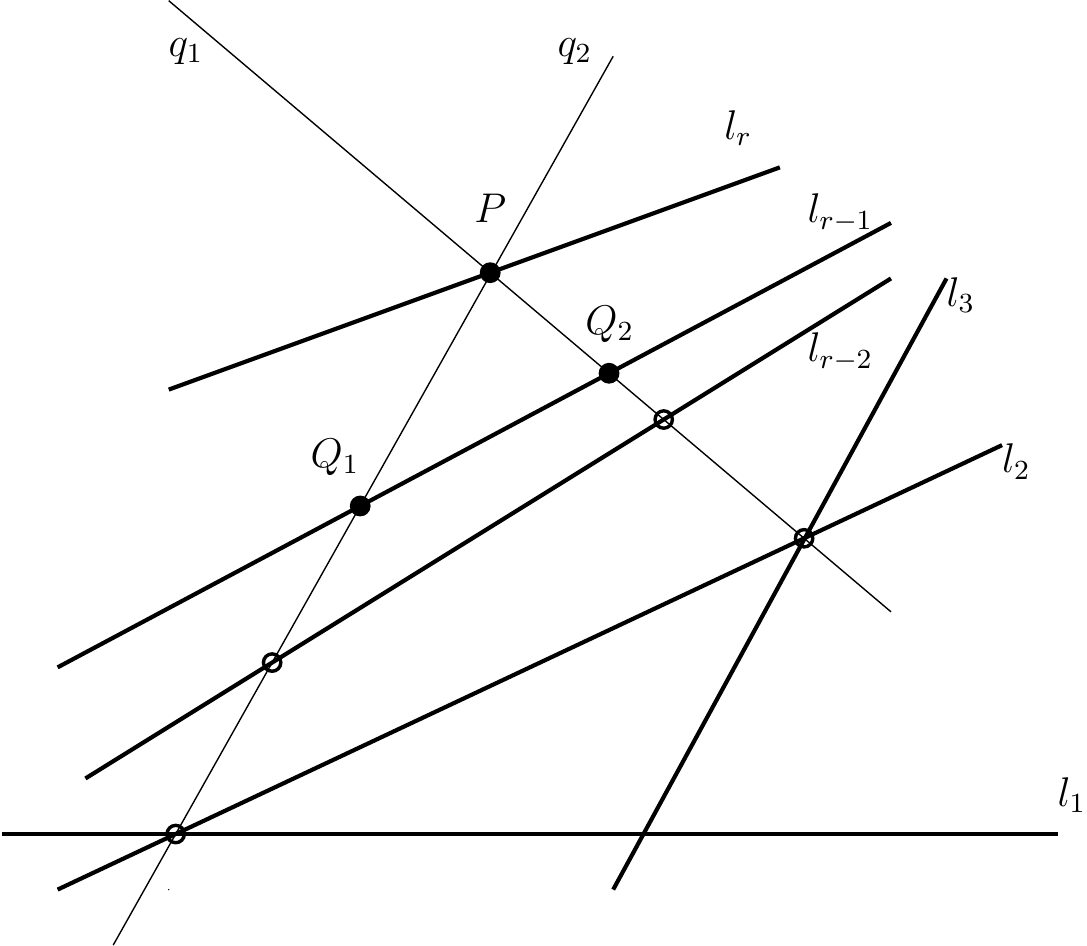}
\caption{The construction of $A_r$}
\label{fig:constructionAr}
\end{center}
\end{figure}

\begin{claim}
For $j\le r-1$ we have $$A_r^j=A_r \bigcup \cup_{i=1}^j l_i .$$

\end{claim}

\begin{proof}[Proof of the claim] We prove it by induction on $j$.

\vspace{1mm}

If $j=1$, then suppose that a line $l \neq l_1$ would be infected. $l$ can not be $l_i$ for $i \ge 2$ as $l_i$ contains only $r+1-i$ points at the beginning. Then, as $l_1,l_2,...,l_r$ are in general position and there is no infected point in their pairwise intersections by Property 2, $l$ should contain $1$ infected point from each $l_i$ ($1 \le i \le r$). So $l$ is either $q_1$ or $q_2$. However $|q_i \cap (\cup_{i=1}^r l_i)|=r-1$ ($i=1,2$) by Property 3, so they can not be infected in the first round.

\vspace{1mm}

For $2 \le j \le r-1$ we can apply a similar argument. By induction we know that only the points on $l_1,l_2,...,l_{j-1}$ are infected before the $j^{th}$ round (beside the initially infected points in $A_r$). So a line $l$ that is $l_i$ with $i \ge j+1$ can not be infected, since it contains $r+1-(j+1)$ initially infected points and at most $j-1$ that can come from an intersection with a so far infected line. So using that $j \le r-1$ we have again that $l$ should be either $q_1$ or $q_2$ which can not be as $|q_j \cap (\cup_{i=1}^r l_i)|=r-1$ ($j=1,2$) by Property 3.

\end{proof}

Note also that by similar argument that was used in the proof of the claim above any line that is infected in the $r^{th}$ round should contain $P$. from which these two things follow: 

\vspace{1mm}

1. $A_r^r$ contains an $r$-broom by the conditions on $r$, thus by Proposition \ref{r-broom} $A_r$ percolates.

\vspace{1mm}

2. As $|q_j \cap (\cup_{i=1}^r l_i)|=r-1$ ($j=1,2$) by Property 3, $q_j$ won't be infected in round $r$ and so $q_j\setminus \cup_{i=1}^r l_i$ is disjoint from  $A_r^r$. Thus $A^r_r \neq \Pi_q$.

\vspace{2mm}

By these we are done with the proof.
\end{proof}

\noindent
We also show this result is sharp for as long as $r$ is sufficiently small compared to $q$.

\begin{proposition}\label{maxpercequal}
If $\binom{r}{2}\leq q$  then $T_r(\Pi_q)\leq r+1$.
\end{proposition}
\begin{proof}
As in every round at least one new infected line arises, this is a straightforward corollary of Proposition \ref{r-lines}. 
\end{proof}

Computation by exhaustive search for small values as well as by randomly selected infected sets for larger $q$ shows that for larger values of $r$ the percolation time can get reasonably bigger than $r+1$ (entries in bold are exact maximal values, while a non-bold entry $i$ means there exist a percolating set with percolation time $i$).
\begin{center}
$ 
\begin{array}{|l||c|c|c|c|c|c|c|c|c|c|c|c|c|c|c|c|}
\hline
q= &       3 & 3 & 5 & 5 & 5 & 7 & 7 & 7  & 11 & 11 & 11 & 13 & 17 & 19 \\
\hline
r= &       2 & 3 & 3 & 4 & 5 & 5 & 6 & 7  & 9  & 10 & 11 & 13 & 17 & 19 \\
\hline
T\ge &      {\bf 2} & {\bf 2} & {\bf 3} & {\bf 5} & {\bf 8} & 6 & 9 & \underline{14} & 10 & 15 & 21 & 23 & 24 & 27\\
\hline
\end{array}
$
\end{center}

Note that at this point it is believed but by no means justified  that the higher the infection rate $r$ the longer the slowest percolation lasts. We formulate this as a conjecture:

\begin{conjecture}
If $r_1<r_2$ then $T_{r_1}(\Pi_q)\leq T_{r_2}(\Pi_q)$.
\end{conjecture}

The extremal case $r=q$ seems to be especially interesting as computation suggests that at least for certain values of $q$ (such as $q=7$) $T_q(\Pi_q)\geq 2q$ is attainable.  No reasonable upper bound on $T_r(\Pi_q)$ has been found. 

\section{Critical probability of percolation}

In this section we investigate two random versions of line percolation in $\Pi_q$. We will consider the random subset $\Pi_q(p)$ of points in $\Pi_q$ where every point $P$ of $\Pi_q$ is an element of $\Pi_q(p)$ with probability $p$ independently of any other point of $\Pi_q$. We will determine the threshold function of $\Pi_q(p)$ percolating. Moreover, we obtain a stronger bottleneck result that states that whenever the random subset $\Pi_q(p)$ meets a line in at least $r$ vertices, then it percolates. Let $S_{\Pi_q}=(P_1,P_2,\dots,P_{q^2+q+1})$ be a random permutation of the points of $\Pi_q$ chosen uniformly among all $(q^2+q+1)!$ permutations. Let $\tau_r(\Pi_q)$ denote the random variable of the minimum index $i$ such that there exists a line $\ell \in \Pi_q$ that contains $r$ points from the first $i$ points of $S_{\Pi_q}$ and let $\tau_{perc,r}(\Pi_q)$ denote the random variable of the minimum index $i$ such that the set of the first $i$ points of $S_{\Pi_q}$ percolates. Obviously, $\tau_r(\Pi_q)\le \tau_{perc,r}(\Pi_q)$ holds for every instance of $S_{\Pi_q}$. We say that a sequence $\mathcal{E}_q$ of events holds \textit{with high probability} (w.h.p., in short) if $\mathbb{P}(\mathcal{E}_q)\rightarrow 1$ as $q$ tends to infinity. We will need an auxiliary third random model: $\Pi_q(m)$ is a random $m$-element point set in $\Pi_q$ chosen uniformly among all $\binom{q^2+q+1}{m}$ $m$-tuples. Note that the random set of the first $m$ points in $S_{\Pi_q}$ has the same probability distribution as $\Pi_q(m)$. The following proposition is well-known (see \cite{jlr} Corollary 1.16).

\begin{proposition}
\label{equiv} Let $\mathcal{Q}$ be a monotone property of point sets in $\Pi_q$. Then if $\Pi_q(p)$ possesses $\mathcal{Q}$ w.h.p., so does $\Pi_q(m)$ for $m=p(q^2+q+1)$. Also, $\mathbb{P}(\Pi_q(p) ~\text{posseses}\ \mathcal{Q})\rightarrow 0$ implies $\mathbb{P}(\Pi_q(m) ~\text{posseses}\ \mathcal{Q})\rightarrow 0$.
\end{proposition}

We are ready to state and prove our main results in this section.

\begin{thm}
Let $\omega=\omega(q)$ be an arbitrary positive function tending to infinity. Then for every fixed $r$ the following hold w.h.p.:

\vspace{2mm}

\noindent \textbf{(a)} If $p=q^{-\frac{r+2}{r}}/\omega$, then $\Pi_q(p)$ does not percolate.

\vspace{1mm}

\noindent \textbf{(b)} If $p=q^{-\frac{r+2}{r}}\omega$, then $\Pi_q(p)$ percolates.

\vspace{1mm}

\noindent \textbf{(c)} $\tau_r(\Pi_q)=\tau_{perc,r}(\Pi_q)$ holds for all $r\ge 3$.
\end{thm}

\begin{proof}
For every line $\ell$ of $\Pi_q$, let us introduce the random indicator variable $X_{\ell,r}$ of the event $|\ell \cap \Pi_q(p)| \ge r$, and let us write $X_r=\sum_{\ell}X_{\ell,r}$ the number of lines in $\Pi_q$ that contain at least $r$ points from $\Pi_q(p)$. We will need the following estimates for the expected value of $X_r$:
\[
(q^2+q+1)\binom{q+1}{r}p^r(1-p)^{q+1-r}\le \mathbb{E}(X_r) \le (q^2+q+1)\binom{q+1}{r}p^r
\]

To prove \textbf{(a)} it is enough to prove $\Pi_q(p)$ does not meet any line in at least $r$ points (i.e. $X_r=0$) and therefore percolation does not even start w.h.p.. To this end it is enough to see that $\mathbb{E}(X_r) \rightarrow 0$, which holds as
\[
\mathbb{E}(X_r) \le (q^2+q+1)\binom{q+1}{r}p^r=O(q^{r+2})q^{-\frac{(r+2)r}{r}}/\omega^r=O(\omega^{-r}).
\]

During the proof of \textbf{(b)} we assume $\omega \le \log q$ and we show that if $p=q^{-\frac{r+2}{r}}\omega$, then $X_r>r$ holds w.h.p.. This would imply the statement as then Proposition \ref{r-lines} ensures that percolation happens in at most two rounds. 

The assumption $\omega \le \log q$ will be used to ensure $(1-p)^{q+1-r}\rightarrow 1$ as $q$ tends to infinity. Indeed, obviously $\omega \le \log q$ implies $p \rightarrow 0$ and therefore $1-p=(1+o(1))e^{-p}$. So $(1-p)^{q+1-r}\ge (1-p)^q=(1-o(1))e^{-qq^{-\frac{r+2}{r}}\omega}=(1-o(1))e^{-\omega/q^{2/r}}=1-o(1)$ as $\log q /q^{\frac{2}{r}}\rightarrow 0$. By the estimates from the beginning of the proof, this implies $\mathbb{E}(X_r)=(1-o(1)) (q^2+q+1)\binom{q+1}{r}p^r$.

By Chebyshev's inequality it is enough to see that $\mathbb{E}(X_r)\rightarrow \infty$ and $\sigma(X_r)=o(\mathbb{E}(X_r))$ hold. The former follows as 

\[
\mathbb{E}(X_r) =(1-o(1)) (q^2+q+1)\binom{q+1}{r}p^r=\Omega(q^{r+2})q^{-\frac{(r+2)r}{r}}\omega^r=\Omega(\omega^{r}).
\]
For the latter observe that for any different $\ell_1$ and $\ell_2$ we have:
\[
\mathbb{E}(X_r^2)=\mathbb{E}(X_r)+(q^2+q+1)(q^2+q)\mathbb{E}(X_{\ell_1,r}X_{\ell_2,r})\le \mathbb{E}(X_r)+(q^2+q+1)^2\mathbb{E}(X_{\ell_1,r}X_{\ell_2,r}).
\]
As $\mathbb{E}(X_{\ell_1,r}X_{\ell_2,r}) \le p\binom{q}{r-1}^2p^{2(r-1)}+\binom{q}{r}^2p^{2r}$ holds and $\mathbb{E}(X_r)\rightarrow \infty$ implies $\mathbb{E}(X_r)=o(\mathbb{E}^2(X_r))$, we have
\begin{equation*}
\begin{split}
\sigma(X_r)^2 & =\mathbb{E}(X_r^2)-\mathbb{E}^2(X_r)\\
& \le \mathbb{E}(X_r)+(q^2+q+1)^2\binom{q}{r-1}^2p^{2r-1}+(q^2+q+1)^2\binom{q}{r}^2p^{2r}-\\
& \hskip 1cm -(1-o(1))\left[(q^2+q+1)\binom{q+1}{r}p^r\right]^2\\
& \le (q^2+q+1)^2\binom{q}{r-1}^2p^{2r-1}+o(\mathbb{E}^2(X_r)) 
\\
& \le \frac{r^2}{pq^2}\mathbb{E}^2(X_r)+o(\mathbb{E}^2(X_r))=o(\mathbb{E}^2(X_r)),
\end{split}
\end{equation*}
as $pq^2=\omega q^{-\frac{r+2}{r}}q^2\rightarrow \infty$ since $r\ge 2$.

Finally, we prove \textbf{(c)}. Obviously, $\tau_r(\Pi_q)\le \tau_{perc,r}(\Pi_q)$ holds. By \textbf{(a)}, \textbf{(b)} and Proposition \ref{equiv} we know that $\tau_r(\Pi_q)=\Theta(q^{1-\frac{2}{r}})$ holds w.h.p.. At round 1 of the percolation starting from the set $A$ of the first $\tau_r(\Pi_q)$ points of $S_{\Pi_q}$, the (at least one) line $\ell$ that contains $r$ points from $A$ gets infected. In the second round all other lines $\ell'$ get infected for which $|\ell'\cap A|\ge r-1$ and $\ell'\cap \ell\notin A$. By Proposition \ref{r-lines}, it is enough to show that there are at least $r-1$ such lines. Let us consider the random variable $X_{r-1}$ for $p=\Theta(q^{-\frac{r+2}{r}})$. Using the estimates from the beginning of the proof, we obtain $\mathbb{E}(X_{r-1})=\Theta(q^{r+1}p^{r-1})=\Theta(q^{\frac{2}{r}})$. The calculation for the variance of $X_r$ (now applied to $X_{r-1}$) stays valid and thus we obtain that $X_{r-1}=\Theta(q^{\frac{2}{r}})$ w.h.p.. Let us introduce the random indicator variable $Y_{r,m,P}$ of the event that $$P\in \Pi_q(p) \wedge \exists \ell_1,\ell_2,\dots,\ell_m  \textrm{ with }  P\in \ell_i \ \textrm{ and } \ |\ell_i\cap \Pi_q(p)|\ge r$$ and let $Y_{r,m}=\sum_{P\in \Pi_q}Y_{r,m,P}$. We will bound $\mathbb{E}(Y_{r-1,q^{1/r}})$ using the well-known estimate $\binom{n}{k}\le (\frac{en}{k})^k$ and assuming $p=\Theta(q^{-\frac{r+2}{r}})$.
\[
\mathbb{E}(Y_{r-1,q^{1/r}})\le (q^2+q+1)\binom{q+1}{q^{1/r}}\binom{q}{r-2}^{q^{1/r}}p^{(r-2)q^{1/r}+1} \le 
\]
\[
2q^2p(eq^{1-1/r})^{q^{1/r}}(qp)^{(r-2)q^{1/r}}=O(q^2pe^{q^{1/r}}q^{q^{1/r}(1-1/r-2(r-2)/r)})
\]
If $r\ge 4$, then $1-1/r-2(r-2)/r<0$ holds, and therefore the above expected value tends to 0 and thus $Y_{r-1,q^{1/r}}=0$ w.h.p. This means that the points of $\ell \cap A$ may lie on at most $rq^{1/r}$ lines that contain at least $r-1$ points from $A$, and therefore there are $X_{r-1}-rq^{1/r}=\Theta(q^{\frac{2}{r}})$ lines that get infected in the second round of the percolation.

Finally let us consider the case $r=3$. As any pair of points of $\Pi_q$ defines a line and there is one line $\ell$ that contains exactly three points $P_1,P_2,P_3$ of $A$, any line $\ell'$ defined by $P',P''\in A\setminus \{P_1,P_2,P_3\}$ gets infected latest in the second round of percolation and then Proposition \ref{r-lines} can be used to deduce that $A$ percolates.
\end{proof}
\bibliographystyle{acm}
\bibliography{refs.bib}
\end{document}